\documentclass [11pt] {amsart}

\usepackage{amsmath,amsthm, amsfonts, amssymb, setspace, color, enumerate, breqn}

      \newtheorem{theorem}{Theorem}
      
      \newtheorem{lemma}{Lemma}
      
      \newtheorem{proposition}{Proposition}
      
      \def\N{{\mathbb N}}
      \def\C{{\mathbb C}}

      \def\cFdd'{\mathcal F_{dd'}}

      \def\cK{\mathcal K}
      
      \def\cH{\mathcal H}
      \def\cV{\mathcal V}

      \def\cA{\mathcal A}
      \def\cD{\mathcal D}
      \def\cC{\mathcal C}
      \def\cE{\mathcal E}
      
      \def\cY{\mathcal Y}
      \def\cX{\mathcal X}
      
      \def\cK{\mathcal K}
      \def\cR{\mathcal R}
      \def\cB{\mathcal B}
      \def\cP{\mathcal P}
      \def\cQ{\mathcal Q}
      \def\cL{\mathcal L}
      
      \def\matnd{\mathbb M_n(\mathbb C^d)}
      \def\matd{\mathbb M(\mathbb C^d)}
      
      \def\sW{\mathcal W}
      \def\sA{\mathcal A}
      \def\sB{\mathcal B}
      \def\sC{\mathcal C}
      \def\sD{\mathcal D}
      \def\sAn{\sA^{(n)}}
      \def\sBn{\sB^{(n)}}
      \def\sCn{\sC^{(n)}}
      \def\sDn{\sD^{(n)}}

\begin{document}
\title{Toeplitz Corona and the Douglas Property for Free Functions}
\author[Balasubramanian]{Sriram Balasubramanian}
\address{Department of Mathematics \& Statistics\\
  Indian Institute of Science Education and Research, Kolkata, India.}
\email{bsriram80@yahoo.co.in, bsriram@iiserkol.ac.in}



\maketitle
\begin{abstract}
  The well known Douglas Lemma says that for operators $A,B$ on Hilbert space that
  $AA^*-BB^*\succeq 0$ implies $B=AC$ for some contraction operator $C$.  
  The result carries over directly to classical operator-valued Toeplitz operators by simply 
  replacing operator by Toeplitz operator.  Free functions generalize the notion
  of free polynomials and formal power series and trace back to the work of J. Taylor
  in the 1970s.  They are of current interest, in part because of their connections
  with free probability and engineering systems theory. For free functions $a$ and $b$
  on a free domain $\cK$  defined free polynomial inequalities, a sufficient condition
  on the difference $aa^*-bb^*$ to imply the existence a free function $c$ taking contractive
  values on $\cK$ such that $a=bc$ is established. The connection to recent work of Agler and McCarthy
  and their free Toeplitz Corona Theorem is exposited. 
\end{abstract}
\begin{section}{Introduction}
 Free functions can be traced back to the work of Taylor \cite{T1}, \cite{T2} and generalize 
 formal power series which appear in the study of finite automata \cite{Sch}. 
 More recently they have been of interest for their connections with free probablity
  and engineering systems theory, see for instance, \cite{VDN}, \cite{V2}, \cite{V1}, \cite{BGT}, \cite{HKM2}, \cite{KVV}, \cite{AKV}, \cite{P1}, \cite{P2}, \cite{P3}, \cite{P4}, \cite{PT}, \cite{AM}, 
\cite{AM1}, \cite{AM2}, \cite{BM}.\\

 This article provides a conceptually different proof of a result in \cite{AM} 
 of a sufficient condition for the existence of a factorization $b=ac$, for free  functions 
 $a,b$ and a free contractive-valued function $c$ on a free domain determined by free polynomials. As a consquence, the Toeplitz Corona Theorem of \cite{AM} is obtained. For more on the Corona and the Toeplitz-Corona problems, see \cite{AM}, \cite{C}, \cite{CSW}, \cite{L}, \cite{Li}, \cite{Sc}, \cite{TW1}, \cite{TW2}, \cite{DS}.\\

All Hilbert spaces considered in this article are Complex and separable.  Let $\matd$ denote graded set $(\matnd)_n$, 
where $\matnd$ is the set of $d$-tuples $X=(X_1,\dots,X_d)$ of $n\times n$ matrices. Observe that the graded set $\matd$ is 
closed with respect to direct sums and unitary conjugations. More generally,\\

A {\it non-commutative set} $\cL = (\cL(n))_n$ is a graded set where $\cL(n) \subset M_n(\C^d)$ such that for $X \in \cL(m)$, $Y \in \cL(n)$ and a unitary matrix $U \in M_m(\C)$, 
\begin{itemize}
\item  [(i)] $X \oplus Y = (X_1 \oplus Y_1, \dots, X_d \oplus Y_d) \in \cK(m+n)$; and 
\item  [(ii)] $U^*XU = (U^*XU_1, \dots, U^*X_dU) \in \cK(m)$.\\
\end{itemize}

A {\it $B(\cH, \cE)$-valued non-commutative function} defined on the non-commutative set $\cL$ is a function such that for $X \in \cL(m)$, $Y \in \cL(n)$, 
\begin{itemize}
\item [(i)] $f(X) \in B(\cH \otimes \C^m, \cE \otimes \C^m)$.
\item [(ii)] $f(X \oplus Y) = f(X) \oplus f(Y)$
\item [(iii)] $f(S^{-1}XS) = (I_{\cE} \otimes S^{-1}) f(X) (I_{\cH} \otimes S)$ whenever $S \in M_m(\C)$ is invertible and $S^{-1}XS \in \cL(m)$.
\end{itemize}
We will say that such a function is {\it bounded} if $\sup_{n \in \N} E_n < \infty$, where $E_n = \sup_{X \in \cL(n)} \|f(X)\|$. Henceforth we will use the abbreviation ''nc" for "non-commutative". \\


A typical example of an nc function is a {\it free polynomial in the $d$ non-commuting variables $x_1, \dots, x_d$}, which is defined as follows. \\

Let $\mathcal F_d$ be the semigroup of words formed using the $d$-symbols $x_1, \dots, x_d$ and the empty word $\emptyset$ denote the identity element of $\mathcal F_d$. A {\it $B(\C^k)$-valued free polynomial in the non-commuting variables $x_1, \dots, x_d$} is a finite formal sum of the form $\sum_{w \in \mathcal F_d} p_w w$, where $p_w \in B(\C^k)$.  For $w = x_{j_1}x_{j_2} \dots x_{j_m}$, the evaluation of $p$ at $X \in \matnd$, is given by $p(X) = \sum_{w \in \mathcal F_d} p_w \otimes X^w \in B(\C^k \otimes \C^n)$, where $X^w = X_{j_1}X_{j_2} \dots X_{j_m}$. For $0 \in \matnd$, $p(0):= p_{\emptyset} \otimes I_n$. It is easy to see that $p$ is a $B(\C^k)$-valued nc function defined on the nc set $\matd$.\\


   Let $\epsilon$ and $\delta$ be $B(\C^k)$-valued free polynomials in $x_1, \dots, x_d$ and let $\cK$ denote the graded set $((\cK(n))_n$, where 
\begin{equation}
\label{eq:domain}
\cK(n) = \{X \in \matnd\,:\, \exists \,c > 0  \text{ such that } \epsilon(X)\epsilon(X)^* - \delta(X)\delta(X)^* \succ c (I_k \otimes I_n)\}.
\end{equation}
Observe that the graded set $\cK = (\cK(n))_n$ is an nc set.
Throughout this article, we will consider this nc set with the additional assumption that $0 \in \cK(1)$. Our main result is the following. 

\begin{proposition}
\label{prop:gtc}
Let  $\cE_1, \cE_2, \cE_3$ be Hilbert spaces and suppose that $a$ and $b$ are bounded $B(\cE_2, \cE_3)$ and $B(\cE_1, \cE_3)$ valued nc-functions on $\cK$. There exists a $B(\cE_1, \cE_2)$ valued nc-function $f$ such that, for all $n$ and $X\in \cK(n)$, 
\begin{enumerate}[(i)]
 \item  $\|f(X)\| \le 1;$ and 
 \item  $a(X)f(X) = b(X)$,
\end{enumerate}
 if there exists a $B(\ell^2\otimes\mathbb C^k, \cE_3)$-valued nc function $h$ defined on $\cK$ such that 
\begin{equation}
\label{eq:gen}
a(T)a(R)^*-b(T)b(R)^* = h(T)[I_{\ell^2}\otimes (\epsilon(T)\epsilon(R)^* - \delta(T)\delta(R)^*)] h(R)^*
\end{equation} 
for all $n \in \N$ and $R, T \in \cK(n)$.

\end{proposition}

A key ingredient in the proof is the existence of a left-invariant Haar probablity measure on the compact group of unitary matrices in $M_n(\C)$.\\

Observe that if $\epsilon = I_k \emptyset$, where $\emptyset \in \mathcal F_d$ is the empty word, then $\cK$ is the domain $G_{\delta} = (G_{\delta}(n))$  considered in \cite{AM}, where 
\begin{equation}
\label{eq:gdeltan}
G_{\delta}(n) = \{X = (X_1, \dots, X_d) \,:\, \|\delta(X)\| < 1\} \subset \matnd,
\end{equation}
 with the additional assumption that $0 \in G_{\delta}(1)$. The following theorem for the domain  $G_{\delta}$ has been proved in \cite{AM}. 

\begin{theorem}
\label{thm:gtcforam}
Let $\cE_1, \cE_2, \cE_3$ be finite-dimensional Hilbert spaces and suppose that $a$ and $b$ are bounded $B(\cE_2, \cE_3)$ and $B(\cE_1, \cE_3)$ valued nc-functions on $\cK = G_{\delta}$. The following are equivalent.
\begin{enumerate}[(i)]
\item There exists a $B(\ell^{2} \otimes \C^k, \cE_3)$ valued nc-function $h$ defined on $\cK$ such that
\[ a(T)a(R)^*-b(T)b(R)^* = h(T) [I_{\ell^2} \otimes ((I_k \otimes I_n )- \delta(T)\delta(R)^*) ]h(R)^* \]
for all $n \in \N$ and $R, T \in \cK(n)$.
\item There exists a bounded $B(\cE_1, \cE_2)$ valued nc-function $f$ such that $\|f(X)\| \le 1$ and $a(X)f(X) = b(X)$, for all $n \in \N$ and $X \in \cK(n)$.
\item $a(X)a(X)^* - b(X)b(X)^* \succeq 0$ for all $n \in \N$ and $X \in \cK(n)$.
\end{enumerate}
\end{theorem}

It is immediate that a proof of the implication $(i) \implies (ii)$ of Theorem \ref{thm:gtcforam}, follows from Proposition \ref{prop:gtc} by taking $\epsilon = I_k \emptyset$.  Thus the proof given here of Proposition \ref{prop:gtc}, exploiting the Haar measure, provides an alternate and conceptually different proof of $(i)\implies (ii)$ than the one given in \cite{AM}. \\


The article is organized as follows. Section 2 contains some preliminary lemmas that will be used in the sequel. Section 3 contains the proofs of Proposition \ref{prop:gtc} (the main result of this article) and Theorem \ref{thm:gtcforam}.
The article ends with the Toeplitz-Corona theorem of \cite{AM} for the nc domain $\cK = G_{\delta}$ with $0 \in \cK(1)$.
\end{section}

\begin{section}{Preliminaries}
\begin{lemma}
\label{lem:structure}
Let $\cX,\cY$ be separable Hilbert spaces and $W \in B(\cX \otimes \C^n, \cY \otimes \C^n)$. If $W = (I_{\cY} \otimes V) W (I_{\cX} \otimes V^*)$ for all unitaries $V \in M_n(\C)$, 
then there exists an operator $\sW\in B(\cX,\cY)$ such that $W = \sW \otimes I_n.$ 
\end{lemma}
 
\begin{proof}
  The result is an embodiment of the fact that the only $n\times n$ matrices which commute with all $n\times n$ matrices are multiples of the identity. 
  Since $(I_{\cY} \otimes V) W = W(I_{\cX} \otimes V)$ for every unitary $V \in M_n(\C)$, it follows that 
\begin{equation}
 \label{eq:intertwine}
(I_{\cY} \otimes X)W=W(I_{\cX}\otimes X)
\end{equation}
  for every $X \in M_n(\C)$.  Let $\{e_1,\dots,e_n\}$ denote an orthonormal basis for $\C^n$ and
  let $E_{j,k} = e_j e_k^*$ denote the resulting matrix units. Write 
  $W=\sum W_{j,k} \otimes E_{j,k}$ for operators $W_{j,k}:\cX\to \cY$.  Choosing, for $1\le\alpha,\beta\le n$, 
  the matrix $X= e_\alpha e_\beta^*$,  from equation \eqref{eq:intertwine} it follows that
\[
   \sum_k W_{\beta,k} \otimes e_\alpha e_k^*  = \sum_j W_{j,\alpha} \otimes e_j e_\beta^*.
\]
 Hence, 
  $W_{\beta,k}=0$ for $k\ne \beta$,  $W_{j, \alpha}=0$ for $j\ne \alpha$ and  $W_{\alpha,\alpha}=W_{\beta,\beta}$ and the result follows by taking $\mathcal W = W_{\alpha, \alpha}$.
\end{proof}

\begin{lemma}
\label{lem:strcont}
Let $\cH$ be a Hilbert space and suppose $A,B \in B(\cH).$  If $AA^* - BB^* \succ c I$ for some $c > 0,$ then  there exists
 a unique $E \in B(\cH)$ such that $B^* = E^*A^*$ and $\|E^*\| \le 1$. Moreover, if $\cH$ is finite dimensional, then
 $E$ is unique and  $\|E^*\| < 1$. 
\end{lemma}

\begin{proof}
 The Douglas lemma (\cite{D}) implies the existence of a contraction $E$ such that $B=AE$ assuming only that
  $AA^*-BB^*\succeq 0$.  Since the hypotheses imply that $AA^*\succeq cI$ is invertible, in the case that $\cH$ is finite dimensional, 
  it follows that $A$ is invertible and $E=A^{-1}B$ is uniquely determined.  Moreover, since $A(I-EE^*)A^* \succeq c I$ and
 $A$ is invertible, $E$ is a strict contraction. 
\end{proof}

\end{section}

\begin{section}{The Proofs}

Let $G^{(n)} = \{U \in M_n(\C) \, :  \, U^*U = I\}.$ It is well known that $G^{(n)}$ is a compact group with respect to multiplication. Hence there exists a unique left-invariant Haar measure $h^{(n)}$ on $G^{(n)}$ such that $h^{(n)}(G) = 1$ and 
\begin{equation}
\label{eq:haar}
\displaystyle \int_{G^{(n)}} f(U) dh^{(n)}(U) = \displaystyle \int_{G^{(n)}} f(VU) dh^{(n)}(U),
\end{equation}
for all $f \in C(G^{(n)})$, $U, V \in G^{(n)}$. For more details see \cite{C}.\\

Recall the nc set $\cK$ defined in \eqref{eq:domain} and the assumption that $0 \in \cK(1)$. 

%
%

\begin{proof}[Proof of Proposition \ref{prop:gtc}]
Fix $n \in \N$. For all $R, T \in \cK(n)$, rearranging \eqref{eq:gen} yields, 
\begin{equation}
\label{eq:genrearr}
\begin{split}
a(T)a(R)^* +  & h(T) [I_{\ell^2} \otimes \delta(T) \delta(R)^*] h(R)^*\\
  = & h(T) [I_{\ell^2} \otimes \epsilon(T) \epsilon(R)^*] h(R)^* + b(T)b(R)^*.
\end{split}
\end{equation}
Consider the closed subspaces: 
 \[
\mathcal D^{(n)} =\overline{span} \left \{ \begin{bmatrix} (I_{\ell^2} \otimes \delta(R)^*) h(R)^* \\ a(R)^* \end{bmatrix} x \,:\, x \in \cE_3 \otimes \C^n, R \in \cK(n)\right \},
\]
\[\mathcal R^{(n)} = \overline{span} \left \{ \begin{bmatrix} (I_{\ell^2} \otimes \epsilon(R)^*) h(R)^* \\ b(R)^* \end{bmatrix} x \,:\, x \in \cE_3 \otimes \C^n, R \in \cK(n)\right \}
\]
of $(\ell^2 \otimes \C^k \otimes \C^n )\oplus (\cE_2 \otimes \C^n)$ and $(\ell^2 \otimes \C^k \otimes \C^n ) \oplus (\cE_1 \otimes \C^n)$ respectively. \\\\
Let $W^{(n)}: \mathcal D^{(n)} \rightarrow \mathcal R^{(n)}$ be the linear map obtained by extending the map
\[
\begin{bmatrix}(I_{\ell^2}\otimes \delta(R)^*) h(R)^* \\ a(R)^* \end{bmatrix} x \rightarrow \begin{bmatrix} (I_{\ell^2}\otimes \epsilon(R)^*) h(R)^* \\ b(R)^* \end{bmatrix} x
\] linearly to all of $\mathcal D^{(n)}$. It follows from equation  \eqref{eq:genrearr} that $W_n:\mathcal D^{(n)} \rightarrow \mathcal R^{(n)}$ is an isometry
 (and hence the map is indeed well defined). Since the codimensions of $\mathcal D^{(n)}$ and $\mathcal R^{(n)}$ agree, it follows that $W^{(n)}: \cD^{(n)} \rightarrow \cR^{(n)}$ can be extended to a unitary $V^{(n)}$.  Thus 
\[
V^{(n)}:= \begin{pmatrix} A^{(n)} & B^{(n)} \\ C^{(n)} & D^{(n)} \end{pmatrix} \,:\,  (\ell^2 \otimes \C^k \otimes \C^n ) \oplus (\cE_2 \otimes \C^n) \rightarrow  (\ell^2 \otimes \C^k \otimes \C^n) \oplus (\cE_1 \otimes \C^n)
\]
 and satisfies 
\[
\begin{pmatrix} A^{(n)} & B^{(n)} \\ C^{(n)} & D^{(n)} \end{pmatrix} 
  \begin{pmatrix} (I_{\ell^2}\otimes \delta(R)^*) h(R)^* \\ a(R)^* \end{pmatrix} = \begin{pmatrix} (I_{\ell^2}\otimes \epsilon(R)^*) h(R)^* \\ b(R)^* \end{pmatrix}
\]
i.e.
\begin{equation}
\label{eq:first}
\sum_{\ell = 1}^k A^{(n)} (I_{\ell^2} \otimes \delta(R)^*)h(R)^* + B^{(n)}a(R)^* =  (I_{\ell^2} \otimes \epsilon(R)^*)h(R)^*,
\end{equation}
\begin{equation}
\label{eq:second}
C^{(n)}(I_{\ell^2} \otimes \delta(R)^*)h(R)^*  + D^{(n)} a(R)^* = b(R)^*.
\end{equation}
Let $U \in G^{(n)}$. Observe that $U^*RU \in \cK(n)$. Replacing $R$ in equations \eqref{eq:first} and \eqref{eq:second} 
  by $U^*R U$ yields,
\begin{dmath}
\label{eq:third}
 A^{(n)} (I_{\ell^2} \otimes I_k \otimes U^*) (I_{\ell^2} \otimes \delta(R)^*)h(R)^* (I_{\cE_3} \otimes U) 
 + B^{(n)} (I_{\cE_2} \otimes U^*) a(R)^* (I_{\cE_3} \otimes U) 
=  (I_{\ell^2} \otimes I_k \otimes U^*) (I_{\ell^2} \otimes \epsilon(R)^*)h(R)^*  (I_{\cE_3} \otimes U),
\end{dmath}
and
\begin{dmath}
\label{eq:fourth}
C^{(n)}(I_{\ell^2} \otimes I_k \otimes U^*) (I_{\ell^2} \otimes \delta(R)^*)h(R)^* (I_{\cE_3} \otimes U) 
+ D^{(n)} (I_{\cE_2} \otimes U^*) a(R)^* (I_{\cE_3} \otimes U)
  =  (I_{\cE_1} \otimes U^*)b(R)^*(I_{\cE_3} \otimes U).
\end{dmath}
Multiplying equation \eqref{eq:third} on the left by $(I_{\ell^2} \otimes I_k \otimes U)$ and on the right by 
$(I_{\cE_3} \otimes U^*)$ and equation \eqref{eq:fourth} on the left by $(I_{\cE_1} \otimes U)$ and on the left by $(I_{\cE_3} \otimes U^*)$ yields,

\begin{dmath}
\label{eq:fifth}
 (I_{\ell^2} \otimes I_k \otimes U) A^{(n)} (I_{\ell^2} \otimes I_k \otimes U^*) (I_{\ell^2} \otimes \delta(R)^*)h(R)^* 
 + (I_{\ell^2} \times I_k \otimes U) B^{(n)} (I_{\cE_2} \otimes U^*) a(R)^*  
=  (I_{\ell^2} \otimes \epsilon(R)^*)h(R)^* ,
\end{dmath}
and

\begin{dmath}
\label{eq:sixth}
(I_{\cE_1} \otimes U) C^{(n)} (I_{\ell^2} \otimes I_k \otimes U^*)  (I_{\ell^2} \otimes \delta(R)^*)h(R)^* 
+ (I_{\cE_1} \otimes U) D^{(n)} (I_{\cE_2} \otimes U^*) a(R)^*  = b(R)^*.
\end{dmath}

Let $\tilde{A}^{(n)}$, $\tilde{B}^{(n)}$, $\tilde{C}^{(n)}$ and $\tilde{D}^{(n)}$ denote the bounded (in fact, contractive) operators that satisfy 
\begin{equation}
\label{eq:ops}
\begin{split}
&\langle \tilde{A}^{(n)} x, y \rangle = \displaystyle \int_{G^(n)} \langle {A}^{(n)} (I_{\ell^2} \otimes I_k \otimes U^*) x,  (I_{\ell^2}  \otimes I_k \otimes U^*) y \rangle \,dh^{(n)}(U) \\
& \langle \tilde{B}^{(n)} a, b \rangle = \displaystyle \int_{G^(n)} \langle {B}^{(n)} (I_{\cE_2} \otimes U^*) a,  (I_{\ell^2} \otimes I_k \otimes U^*)b \rangle \,dh^{(n)}(U) \\
& \langle \tilde{C}^{(n)} z, w \rangle = \displaystyle \int_{G^(n)} \langle {C}^{(n)} (I_{\ell^2} \otimes I_k \otimes U^*) z,   (I_{\cE_1} \otimes U^*)w \rangle \,dh^{(n)}(U) \\
& \langle \tilde{D}^{(n)} g, h \rangle = \displaystyle \int_{G^(n)} \langle {D}^{(n)} (I_{\cE_2} \otimes U^*)g ,   (I_{\cE_1} \otimes U^*) h\rangle \,dh^{(n)}(U) \\
\end{split}
\end{equation}
for all $x,y,b,z \in \ell^2 \otimes \C^k \otimes \C^n$; $a,g \in \cE_2 \otimes \C^n$; $w, h \in \cE_1 \otimes \C^n$. 
Moreover,
For $x \in \cE_3 \otimes \C^n$ and $y \in \ell^2 \otimes \C^k \otimes \C^n$, $u \in \cE_3 \otimes \C^n$ and $v \in \cE_1 \otimes \C^n$, it follows from equations \eqref{eq:ops}, \eqref{eq:fifth} and \eqref{eq:sixth} that 
\begin{dmath}
\label{eq:seventh}
 \left \langle   [\tilde{A}^{(n)}  (I_{\ell^2} \otimes \delta(R)^*)h(R)^* + \tilde{B}^{(n)} a(R)^* ]x,y \right \rangle  
= \int_{G^{(n)}} \left \langle [ (I_{\ell^2} \otimes I_k \otimes U) A^{(n)}(I_{\ell^2} \otimes I_k \otimes U^*) (I_{\ell^2} \otimes \delta(R)^*)h(R)^* 
+ (I_{\ell^2} \otimes I_k \otimes U) B^{(n)} (I_{\cE_2} \otimes U^*) a(R)^* ] x,y \right \rangle dh^{(n)}(U)  = 
\int_{G^{(n)}} \left \langle  (I_{\ell^2} \otimes \epsilon(R)^*)h(R)^* x, y \right\rangle dh^{(n)}(U) =  \left \langle  (I_{\ell^2} \otimes \epsilon(R)^*)h(R)^* x, y \right\rangle
\end{dmath}
 
as well as

\begin{dmath}
\label{eq:eighth}
\left \langle [ \tilde{C}^{(n)} (I_{\ell^2} \otimes \delta(R)^*)h(R)^* + \tilde{D}^{(n)} a(R)^* ] u,v \right \rangle = 
\int_{G^{(n)}} \left \langle [ (I_{\cE_1} \otimes U) C^{(n)} (I_{\ell^2} \otimes I_k \otimes U^*) (I_{\ell^2} \otimes \delta(R)^*)h(R)^* + (I_{\cE_1} \otimes U) D^{(n)} (I_{\cE_2} \otimes U^*) a(R)^* ] u, v \right \rangle \, dh^{(n)}(U) = 
\int_{G^{(n)}} \langle   b(R)^*u, v \rangle\, dh^{(n)}(U). = \langle   b(R)^*u, v \rangle 
\end{dmath} 
Equations \eqref{eq:seventh} and \eqref{eq:eighth} together imply that 
\[
\begin{pmatrix} \tilde{A}^{(n)} & \tilde{B}^{(n)} \\ \tilde{C}^{(n)} & \tilde{D}^{(n)} \end{pmatrix} \begin{pmatrix} (I_{\ell^2} \otimes \delta(R)^*) h(R)^* \\ a(R)^* \end{pmatrix} = \begin{pmatrix} (I_{\ell^2} \otimes \epsilon(R)^* )h(R)^* \\ b(R)^* \end{pmatrix}.
\]
Also, observe that $\begin{pmatrix} \tilde{A}^{(n)} & \tilde{B}^{(n)} \\ \tilde{C}^{(n)} & \tilde{D}^{(n)} \end{pmatrix}$ is a contraction.
Lastly, for $V \in G^{(n)}$, the left invariance property of the Haar measure $h$ implies that $\tilde{A}^{(n)}$, $\tilde{B}^{(n)}$, $\tilde{C}^{(n)}$ and $\tilde{D}^{(n)}$ are invariant under conjugation by $I \otimes V$ and hence 
\begin{align*}
&\tilde{A}^{(n)} = (I_{\ell^2} \otimes I_k \otimes V) \tilde{A}^{(n)} (I_{\ell^2} \otimes I_k \otimes V^*) \\
& \tilde{B}^{(n)} = (I_{\ell^2} \otimes I_k \otimes V)  \tilde{B}^{(n)} (I_{\cE_2} \otimes V^*) \\
& \tilde{C}^{(n)} = (I_{\cE_1} \otimes V) \tilde{C}^{(n)}(I_{\ell^2} \otimes I_k \otimes V^*) \\
& \tilde{D}^{(n)} = (I_{\cE_1} \otimes V) \tilde{D}^{(n)} (I_{\cE_2} \otimes V^*). 
\end{align*}


It follows from Lemma \ref{lem:structure} that 
 there exists bounded operators $\sAn$, $\sBn,$ $\sCn,$ and $\sDn$ such that 
$\tilde{A}^{(n)} = \sAn \otimes I_n$, $\tilde{B}^{(n)} = \sBn \otimes I_n$, 
  $\tilde{C}^{(n)}  = \sCn \otimes I_n$ and $\tilde{D}^{(n)} = \sDn \otimes I_n$,
 where  $\sAn \in B (\ell^2 \otimes \C^k)$, $\sBn \in B(\cE_2, \ell^2 \otimes \C^k)$, $\sCn \in B(\ell^2 \otimes \C^k, \cE_1)$ and $\sDn \in B(\cE_2,\cE_1)$. Moreover, 
\[
   \begin{pmatrix} \sAn & \sBn \\ \sCn & \sDn \end{pmatrix}: (\ell^{2} \otimes \C^k) \oplus \cE_2 \rightarrow (\ell^{2} \otimes \C^k)\oplus \cE_1
\] is a contraction.

Let $\cH = (\ell^{2} \otimes \C^k) \oplus \cE_2$ and $\cE = (\ell^{2} \otimes \C^k) \oplus \cE_1$. Observe that $\cH \oplus \cE$ is separable.
At this point, it has been proved that there exists an operator $\cV \in B(\cH,\cE)$ such that $\|\cV\|\le 1$ and 
\begin{equation}
 \label{eq:THEeq}
   \cV\otimes I_n \, 
  \begin{pmatrix} (I\otimes \delta(R)^*) h(R)^* \\ a(R)^* \end{pmatrix} = \begin{pmatrix} (I\otimes \epsilon(R)^*) h(R)^* \\ b(R)^* \end{pmatrix}.
\end{equation}
 Let 
\[
L_n = \left\{  \begin{pmatrix}  0 & 0 \\ \cV & 0 \end{pmatrix} \,:\, 
\|\cV\| \le 1 \text{ and } (\cV \otimes I_n) \text{ solves \eqref{eq:THEeq}}  \right\} \subset  B(\cH \oplus \cE).
\]
The argument above implies that $L_n \neq \emptyset$ for each $n \in \N$. 
 It is also the case that $L_n$ is a WOT-closed subset of the WOT-compact unit ball of $B(\cH \oplus \cE)$. Thus $L_n$ is WOT-compact for each $n \in \N$.
 Moreover since $0 \in \cK(1)$, it follows that $L_n \supset L_{n+1}$. 
By the nested intersection property of compact sets,
  $\displaystyle \bigcap_{n \in \N} L_n$ is non-empty. Say $\begin{pmatrix} 0 & 0 \\ V & 0 \end{pmatrix} \in \displaystyle \bigcap_{n \in \N} L_n$, where $V = \begin{pmatrix} A & B \\ C & D \end{pmatrix}$ with $A \in B(\ell^{2} \otimes \C^k)$, $B \in B(\cE_2, \ell^{2} \otimes \C^k)$, $C\in B(\ell^{2} \otimes \C^k, \cE_1)$ and $D \in B(\cE_2, \cE_1)$.\\ 


For all $n \in \N$ and $R \in \cK(n)$, we have,
\begin{align}
\label{eq:universalsolution1}
&(A \otimes I_n) (I_{\ell^2} \otimes\delta(R)^*) h(R)^* + (B \otimes I_n)a(R)^* = (I_{\ell^2} \otimes\epsilon(R)^*)h(R)^* \\
\label{eq:universalsolution2}
&(C \otimes I_n) (I_{\ell^2} \otimes \delta(R)^*)h(R)^* + (D \otimes I_n) a(R)^* = b(R)^*.
\end{align}

By Lemma \ref{lem:strcont}, for each $n \in \N$ and $R\in\cK(n)$ there exists a uniquely determined strict contraction $\gamma(R) \in B(\C^k \otimes \C^n)$ such that
\begin{equation}
\label{eq:z}
\delta(R)^* = \gamma(R)^* \epsilon(R)^*.
\end{equation}


Since $\|A \otimes I_n\| \le 1$ and $\| \gamma(R)^*\| <1$, rearranging equation \eqref{eq:universalsolution1} and using \eqref{eq:z} yields,
\begin{equation}
\label{eq:imp1}
(I_{\ell^2} \otimes \epsilon(R)^*)h(R)^* = \{I_{\ell^{2}} \otimes I_k \otimes I_n - (A \otimes I_n) (I_{\ell^2} \otimes \gamma(R)^*)\}^{-1}(B \otimes I_n)a(R)^*.
\end{equation}
Using \eqref{eq:imp1} and \eqref{eq:z} in \eqref{eq:universalsolution2} yields,
\begin{dmath}
[(C \otimes I_n) (I_{\ell^2} \otimes \gamma(R)^*) \{I_{\ell^{2}} \otimes I_k \otimes I_n - (A \otimes I_n) (I_{\ell^2} \otimes \gamma(R)^*)\}^{-1}(B \otimes I_n) + (D \otimes I_n)]a(R)^* = b(R)^*.
\end{dmath}
For $n \in \N$, $R \in \cK(n)$, define the function $f$ on $\cK$ by 
\begin{dmath}
f(R) = [(C \otimes I_n) (I_{\ell^2} \otimes \gamma(R)^*) \{I_{\ell^{2}} \otimes I_k \otimes I_n - (A \otimes I_n) (I_{\ell^2} \otimes \gamma(R)^*)\}^{-1}(B \otimes I_n) + (D \otimes I_n)]^*
\end{dmath}
Thus $f$ is a $B( \cE_1,\cE_2)$-valued graded function which satisfies $a(R)f(R) = b(R)$. It is also easy to see that $f$ preserves direct sums.\\

Finally, to show that $f$ is an nc function, suppose $R\in\cK(n)$ and $S$ is an invertible $n\times n$ matrix such that $S^{-1}RS\in\cK(n)$. 
 We need to show that $f(S^{-1}RS) =  (I_{\cE_2} \otimes S^{-1})f(R)(I_{\cE_1} \otimes S)$.  Observe that $\gamma(R)^*$ is uniquely determined by \eqref{eq:z}, since $\epsilon(R)^*$ is invertible.
   From the form of $f$, it is enough to show $\gamma(S^{-1}RS) =  (I_k \otimes S^{-1})\gamma(R)(I_k \otimes S).$  To this end, observe that,
\begin{dmath}
 (I_k \otimes S^*) \delta(R)^* (I_k \otimes {(S^{*})}^{-1}) = \delta(S^{-1} R S)^* = \gamma (S^{-1} R S)^* \epsilon( S^{-1} R S)^* = \gamma(S^{-1} R S)^* (I_k \otimes S^*) \epsilon(R)^* (I_k \otimes {(S^{*})}^{-1}).
\end{dmath}
Thus
\[
 (I_k \otimes S^*) \gamma(R)^* \epsilon(R)^* (I_k \otimes {(S^{*})}^{-1}) = \gamma(S^{-1} R S)^* (I_k \otimes S^*) \epsilon(R)^* (I_k \otimes {(S^{*})}^{-1}).
\]
Since $\epsilon(R)^* (I_k \otimes {(S^{*})}^{-1})$ is invertible, taking adjoints, it follows that 
\[
(I_k \otimes S^{-1}) \gamma(R) (I_k \otimes S) = \gamma(S^{-1}RS).
\]
The proof is complete if we show that $\|f(R)\| \le 1$ for every $n \in \N$ and $R \in \cK(n)$.
Recall that for all $n \in \N$, $V \otimes I_n = \begin{pmatrix}  \cA  & \cB \\ \cC &  \cD \end{pmatrix} =
\begin{pmatrix}  A \otimes I_n & B \otimes I_n \\ C \otimes I_n & D \otimes I_n \end{pmatrix} $ is a contraction. Thus there exists bounded operators $\cP$ and $\cQ$ such that 
\begin{displaymath}
\label{eq:matrix}
\begin{pmatrix}  \cP^* \cP & \cP^* \cQ \\ \cQ^* \cP & \cQ^* \cQ \end{pmatrix} = 
\begin{pmatrix} I_{\ell^2 \otimes \C^k \otimes \C^n} & 0 \\ 0 & I_{{\cE_2} \otimes \C^n}  \end{pmatrix} - 
\begin{pmatrix}  \cA^*  & \cC^* \\\cB^* &  \cD^* \end{pmatrix} 
\begin{pmatrix}  \cA  & \cB \\ \cC &  \cD \end{pmatrix} \succeq 0.
\end{displaymath}

For notational convenience, let $\Gamma(R) := (I_{\ell^2} \otimes \gamma(R)^*)$, $\Delta(R) := (I_{\ell^2} \otimes I_k \otimes I_n - \cA \Gamma(R))$ and $\Phi(R) := \Delta(R)^{-1}$. We have $f(R)^* = \cD + \cC \Gamma(R) \Phi(R) \cB$. 
Using equation \eqref{eq:matrix}, for $n \in \N$ and $X \in \cK(n)$, we have 

\begin{align*}
(I_{\cE_2} \otimes I_n) - f(R)f(R)^* &= (I_{\cE_2} \otimes I_n) - \cD^* \cD - \cB^* \Phi(R)^* \Gamma(R)^* \cC^* \cD \\
& - 
\cD^* \cC \Gamma(R) \Phi(R) \cB  - \cB^* \Phi(R)^* \Gamma(R)^* \cC^* \cC \Gamma(R) \Phi(R) \cB \\\\
&= \cQ^* \cQ + B^*B + \cB^* \Phi(R)^* \Gamma(R)^*(\cA^* \cB + \cP^* \cQ) \\
&+ (\cB^* \cA + \cQ^* \cP) \Gamma(R) \Phi(R) \cB \\
&- \cB^* \Phi(R)^* \Gamma(R)^* (I - \cA^* \cA - \cP^* \cP) \Gamma(R) \Phi(R) \cB \\\\
&= \cB^* \Phi(R)^* [ \Delta(R)^* \Delta(R)  + \Gamma(R)^* \cA^* \Delta(R) + \Delta(R)^* \cA \Gamma(R) \\
&- \Gamma(R)^*(I - \cA^* \cA) \Gamma(R)] \Phi(R) \cB \\
&+ \cQ^* \cQ + \cB^* \Phi(R)^*\Gamma(R)^* \cP^* \cQ + \cQ^* \cP \Gamma(R) \Phi(R) \cB \\
&+ \cB^* \Phi(R)^* \Gamma(R)^* \cP^* \cP \Gamma(R) \Phi(R) \cB \\\\
&= \cB^* \Phi(R)^* [I - \Gamma(R)^* \Gamma(R)] \Phi(R) \cB \\
&+ (\cQ + \cP \Gamma(R) \Phi(R) \cB)^*(\cQ + \cP \Gamma(R) \Phi(R) \cB)\\\\
&\succeq 0.
\end{align*}
\end{proof}



\begin{proof}[Proof of Theorem \ref{thm:gtcforam}]
(i) implies (ii): Follows from Proposition \ref{prop:gtc}, by letting $\epsilon = I_k \emptyset$.\\\\
(ii) implies (iii): Observe that for each $n \in \N$ and $X \in \cK(n)$, 
\begin{dmath}
a(X)a(X)^* - b(X)b(X)^* = 
a(X)a(X)^* - a(X)f(X)f(X)^*a(X)^* = a(X)(I_{\cE_2} \otimes I_n - f(X)f(X)^*)a(X)^* \succeq 0.
\end{dmath}
(iii) implies (i): This is the content of Theorem 7.10 in \cite{AM}.\\
\end{proof}

Recall the non-commutative set $G_{\delta} = (G_{\delta}(n))_n$ from \eqref{eq:gdeltan}. The following is the Toeplitz-Corona theorem of \cite{AM} for the non-commutative domain $ G_{\delta} = (G_{\delta}(n))$ with the assumption that $0 \in G_{\delta}(1)$. Observe that certain well-known non-commutative domains, for example, the non-commutative polydisc, can be realized as such $G_{\delta}$, for suitable $\delta$.

\begin{theorem}
Let $a_1, \dots, a_{\ell}$ be bounded $\C$-valued nc-functions defined on $G_{\delta}$ and $\mu >0$. If for all $n \in \N$ and $R \in G_{\delta}(n)$, $\sum_{i=1}^{\ell} a_i(R)a_i(R)^* \succeq \mu^2 I_n$, then there exists $\C$-valued nc functions $g_1, \dots, g_{\ell}$ defined on $G_{\delta}$ such that $\sum_{i=1}^{\ell} a_i(R)g_i(R) = I_n$ for each $n \in \N$ and $R \in G_{\delta}(n)$. Moreover the $B(\C,\C^{\ell})$ valued nc function $g$ satisfies  satisfies $\|g(R)\| \le \frac{1}{\mu}$ for all $n \in \N$ and $R \in G_{\delta}(n)$, where $g(R) = e_1 \otimes g_1(R) + \dots e_{\ell} \otimes g_j(R)$ and $e_1, e_2, \dots, e_{\ell}$ are the standard unit (column) vectors in $\C^{\ell}$. 
\end{theorem}

\begin{proof}
Letting $\cE_1 = \cE_3 = \C$ and $\cE_2 = \C^{\ell}$, $a(R) = e_1^* \otimes a_1(R) + \dots + e_{\ell}^* \otimes a_{\ell}(R)$ and $b(R) = \mu I_n$ for $R \in G_{\delta}(n)$ in Theorem \ref{thm:gtcforam}, the hypothesis becomes $a(R)a(R)^* - b(R)b(R)^* \succeq 0$. Theorem 
\ref{thm:gtcforam} now implies that there exists a $B(\C,\C^{\ell})$ valued nc function $f$ such that 
$\|f(R)\| \le 1$ and 
\begin{equation}
\label{eq:af=b}
[e_1^* \otimes a_1(R) + \dots e_{\ell}^* \otimes a_j(R)]f(R) = \mu I_n.
\end{equation}
Choose $\C$-valued nc functions $f_1, \dots, f_{\ell}$ such that $f(R) = e_1 \otimes f_1(R) + \dots e_{\ell} \otimes f_{\ell}(R)$. Using this in equation \eqref{eq:af=b} yields,
\[
 \sum_{i=1}^{\ell} a_i(R)f_i(R) = \mu I_n.
\]
Taking $g_i = \frac{1}{\mu} f_i$; $i = 1, 2, \dots, {\ell}$, completes the proof. 
\end{proof}

%
\end{section}

\section{Free Spectrahedra}
  Let $\Lambda$ denote a linear $r\times r$ matrix-valued nc polynomial,
\[
  \Lambda(x) = \sum_{j=1}^g A_j x_j,
\]
 where the $A_j$ are $r\times r$ matrices.   The corresponding {\it linear pencil} is the expression
\[
  L(x)= I-\Lambda(x)-\Lambda^*(x),
\]
 where $\Lambda^*$ is the formal adjoint of $\Lambda$ determined by,
\[
 \Lambda^*(X) =\Lambda(X)^*
\]
 for tuples $X=(X_1,\dots,X_d)$ of $n\times n$ matrices.   In this case the graded set
  $\cK = (\cK(n))_n$ is known as a {\it free (non-commutative) spectrahedron} (See \cite{HKM1}).  A bit of algebra shows
\[
   L(x)=(I-\Lambda)(x) \, (I-\Lambda)(x)^* -\Lambda(x)\Lambda(x)^*.
\]
\newline
\textbf{Acknowledgement:} I would like to thank Prof. Scott McCullough for several discussions and many helpful suggestions.

\end{document}